\documentclass[12pt]{amsart}

\newtheorem{theorem}{Theorem}

\newtheorem{corollary}{Corollary}
\newtheorem{lemma}{Lemma}

\theoremstyle{definition}
\newtheorem{remark}{Remark}

\author{An.An.\,Novikov, O.E.\,Tikhonov}
\title{Characterization of central elements by inequalities}

\begin{document}

\maketitle

{\tiny Kazan Federal University, Kremlievskaia st. 18, 420008, Kazan, Russia

e-mail: a.hobukob@gmail.com, oleg.tikhonov@kpfu.ru

keywords: von Neumann algebra, C*-algebra, center of algebra, characterization, inequalities.

subclass: 46L05, 46L10, 47C15.

Partially supported by Russian Foundation for Basic Research, Grant 14-01-31358.}

{\bf Abstract:}
We propose a list of inequalities which characterize central elements in von Neumann algebras and C*-algebras.

\vspace{10pt}

In this paper we study the possibility to distinguish central elements of operator algebras among all positive elements, by satisfying certain inequalities for positive functionals. A dual problem of characterizing tracial property by inequalities was considered, for instance,  in \cite{Bik13, Bik10, Bik11, DT, Gardner, P&S, STS, Tikh-pos}, and we apply some machinery of those papers. 

Throughout the note,  $\mathcal{M}$ stands for a von Neumann algebra, $\mathcal{M}^{sa}$, $\mathcal{M}^+$, and $\mathcal{M}^\text{pr}$ denote the selfadjoint part, the positive part, and the set of all projections in $\mathcal{M}$, respectively. $\mathcal{Z}$ denotes the center of $\mathcal{M}$ and $\mathbf{1}$ denotes the identity operator.
Let $\mathcal{M}_*$ denote the space of all normal functionals on $\mathcal{M}$, $\mathcal{M}_*^h$ and $\mathcal{M}_*^+$ denote its Hermitian and positive parts. We will use standard notation for multiplication of a functional $\varphi$ by an operator $x$, namely, $x\varphi$, $\varphi x$ and $x \varphi x$ denote the linear functionals 
$y \mapsto \varphi(xy)$, $y\mapsto \varphi(yx)$ and $y\mapsto \varphi(xyx)$, respectively. Recall that a selfadjoint unitary operator in a Hilbert space is said to be a \emph{symmetry}.  

The proof of the following lemma is adapted from \cite[Lemma 1]{STS}.
  
\begin{lemma}\label{lemma1}
Let $a \in \mathcal{M}^+$. If the inequality $\varphi(sas)<\varphi(a)$ holds for some symmetry $s \in \mathcal{M}$ 
and some positive normal functional $\varphi \in \mathcal{M}_*^+$, then there exist positive normal functionals $\psi_1, \psi_2 \in \mathcal{M}_*^+$, such that $|\psi_1-\psi_2|(a)>\psi_1(a)+\psi_2(a)$.
\end{lemma}

\begin{proof}
For $\lambda>0$ define normal functionals 
$
\psi^{\lambda}_1=\lambda s\varphi s + \varphi s + s\varphi + \lambda ^{-1}\varphi$ and
$\psi^{\lambda}_2=\lambda s\varphi s - \varphi s - s\varphi + \lambda ^{-1} \varphi .
$
Since $\psi_1^\lambda = \lambda \, v_1^\lambda \, \varphi \,v_1^\lambda$ with 
$v_1^\lambda = s + \lambda ^{-1} \mathbf{1}$ and 
$\psi_2^\lambda = \lambda \, v_2^\lambda \, \varphi \,v_2^\lambda$ with $v_2^\lambda = s - \lambda ^{-1} \mathbf{1}$, 
those functionals are positive. 

Let us demonstrate that $|\psi_1^\lambda - \psi_2^\lambda|= 2\varphi + 2s\varphi s$. 
Clearly, $|\psi_1^\lambda - \psi_2^\lambda|=2|\varphi s + s \varphi|$. 
Observe that $\|\varphi s + s \varphi\|=(\varphi s+ s \varphi)(s)$ since  
$\| \varphi s + s \varphi \| \le \| \varphi s \| + \| s \varphi \| \le 2 \| \varphi \| =
 2 \varphi (\mathbf{1})$ and
$\| \varphi s + s \varphi \| \ge |(\varphi s + s \varphi)(s) | = 
(\varphi s + s \varphi )(s) = 2 \varphi (\mathbf{1})$.
By the construction of the absolute value of normal functional \cite[the proof of Theorem III.4.2]{Takesaki}, we have $|\psi_1^\lambda - \psi_2^\lambda|= 
2 | \varphi s + s \varphi | = 2 s(\varphi s + s \varphi ) = 2(\varphi + s \varphi s )$.

With the notation $\varepsilon=(\varphi(a)-\varphi(sas))/{\varphi(a)}$, we obtain by a straightforward calculation
$|\psi_1^\lambda - \psi_2^\lambda|(a) = 2(2-\varepsilon)\varphi(a)$ as well as
$(\psi_1^\lambda+\psi_2^\lambda)(a)=(\lambda(1-\varepsilon)+\lambda^{-1})\varphi(a).$
Since $\min\limits_{\lambda>0}\{\lambda(1-\varepsilon)+\lambda^{-1} \}=2\sqrt{1-\varepsilon}$ and $2-\varepsilon>\sqrt{1-\varepsilon}$ for any $\varepsilon\in(0,1)$, 
there exists $\lambda_0 >0$ such that 
$2(2-\varepsilon)>\lambda_0(1-\varepsilon)+\lambda_0^{-1}$.
Therefore $|\psi_1^{\lambda_0} - \psi_2^{\lambda_0}|(a)>(\psi_1^{\lambda_0}+\psi_2^{\lambda_0})(a)$.
\end{proof}

\begin{theorem}
For $a \in \mathcal{M}^+$ the following conditions are equivalent:
\begin{itemize}
 
\item [$(i)$]
$a$ lies in the center $\mathcal{Z}$ of $\mathcal{M}$;
 
\item [$(ii)$] 
$pap \le a$ for each $p \in \mathcal{M}^\textrm{pr}$;    

\item [$(iii)$]
$\varphi^+ (a) \le \varphi_1 (a)$ for each $\varphi \in \mathcal{M}_*^h$ and any decomposition $\varphi = \varphi_1 - \varphi_2$ with $\varphi_1 , \varphi_2 \in \mathcal{M}_*^+$;   

\item [$(iv)$]
the mapping $\varphi \mapsto \varphi^+(a)\ (\varphi \in \mathcal{M}^h_*)$ is monotone, 

i.\,e., $\varphi , \psi \in \mathcal{M}_*^h$, $\varphi \le \psi$ imply $\varphi^+ (a) \le \psi^+ (a)$; 

\item [$(v)$]
the mapping $\varphi \mapsto \varphi^+(a)\ (\varphi \in \mathcal{M}^h_*)$ is subadditive, 

i.\,e., $(\varphi + \psi )^+ (a) \le \varphi^+ (a) + \psi^+ (a)$ for all $\varphi , \psi \in \mathcal{M}_*^h$; 

\item [$(vi)$]
the mapping $\varphi \mapsto |\varphi |(a)\ (\varphi \in \mathcal{M}^h_*)$ is subadditive, 

i.\,e., $|\varphi + \psi |(a) \le |\varphi |(a) + |\psi |(a)$ for all $\varphi , \psi \in \mathcal{M}_*^h$.
\end{itemize}
\end{theorem}

\begin{proof}
$(i) \Rightarrow (ii)$ 
For $a \in \mathcal{Z}^+$ and $p \in \mathcal{M}^\textrm{pr}$, one has $a = pap + (\mathbf{1} -p)a(\mathbf{1} -p) \ge pap$. 

\smallskip
$(ii) \Rightarrow (iii)$ 
For $\varphi \in \mathcal{M}_*^h$, let $p$ be the support projection of $\varphi^+$ (see \cite[Section III.4]{Takesaki}).     
Then $\varphi^+ (a) = (p \varphi p) (a) \le (p (\varphi + \varphi_2) p) (a) = 
(p \varphi_1 p) (a) = \varphi_1 (pap)
\le \varphi_1 (a)$.

\smallskip
$(iii) \Rightarrow (iv)$ 
This follows from the equality
$\varphi = \psi^+ - (\psi^- + (\psi - \varphi ))$.

\smallskip
$(iv) \Rightarrow (v)$ 
For $\varphi , \psi \in \mathcal{M}_*^h$, one has $\varphi \le \varphi^+$ and $\psi \le \psi^+$, hence $\varphi + \psi \le \varphi^+ + \psi^+$. Then the condition $(iv)$ entails  
$(\varphi + \psi )^+ (a) \le \varphi^+ (a) + \psi^+ (a)$. 

\smallskip
$(v) \Rightarrow (vi)$ 
By the condition $(v)$, for $\varphi , \psi \in \mathcal{M}_*^h$, it holds 
$(\varphi + \psi )^+ (a) \le \varphi^+ (a) + \psi^+ (a)$. 
Also, $(\varphi + \psi )^- (a) = (-\varphi - \psi )^+ (a) \le 
(-\varphi )^+ (a) + (-\psi)^+ (a) = \varphi^- (a) + \psi^- (a)$. 
Hence $|\varphi + \psi |(a) = (\varphi + \psi )^+ (a) + (\varphi + \psi )^- (a) \le 
\varphi^+ (a) + \psi^+ (a) + \varphi^- (a) + \psi^- (a) = |\varphi |(a) + |\psi |(a)$.  

\smallskip
$(vi) \Rightarrow (i)$ 
If an operator $a \in \mathcal{M}^+$ satisfies $(vi)$, then it follows from Lemma \ref{lemma1} 
that $\varphi(sas)=\varphi(a)$ for each symmetry $s$ in $\mathcal{M}$ and any $\varphi$ in $\mathcal{M}_*^h$. 
Therefore $a=sas$ for each symmetry $s$ in $\mathcal{M}$. 
As it is easily seen, the latter implies that $a$ commutes with each projection in $\mathcal{M}$ and therefore lies in $\mathcal{Z}$.    
\end{proof}

\begin{remark}
Clearly, for $a \in \mathcal{M}^+$, $(v)$ is equivalent to the condition

\begin{itemize}
\item [$(vii)$] 
{\it the mapping $\varphi \mapsto \varphi^+(a)\ (\varphi \in \mathcal{M}^h_* )$ is convex.}
\end{itemize}
As well, $(vi)$ is equivalent to the conditions
\begin{itemize}

\item [$(viii)$] 
{\it
the mapping $\varphi \mapsto |\varphi|(a)\ (\varphi \in \mathcal{M}^h_* )$ is convex.}
\end{itemize}
It follows from Theorem 1 that for $a \in \mathcal{M}^+$ each of the conditions $(vii)$, $(viii)$ is equivalent to 
$a \in \mathcal{Z}$. 
\end{remark}

\begin{corollary} \label{cor1}
For $a \in \mathcal{M}^+$, each of the conditions $(i)$ -- $(viii)$ is equivalent to each of the following conditions:
\begin{itemize}
 
\item [$(ix)$] 
$|\varphi|(a)=\|a^\frac{1}{2}\varphi a^\frac{1}{2}\|$ for all $\varphi$ in $\mathcal{M}_*$;  

\item [$(x)$]
the mapping $\varphi \mapsto |\varphi |(a)$ is subadditive on $\mathcal{M}_*$, 

i.\,e., $|\varphi + \psi |(a) \le |\varphi |(a) + |\psi |(a)$ for all $\varphi , \psi \in \mathcal{M}_*$.
\end{itemize}

\end{corollary}

\begin{proof} 
Let $a \in \mathcal{Z}^+$. For $\varphi \in \mathcal{M}_*$, let $\varphi = u |\varphi|$ be the polar decomposition \cite[Section III.4]{Takesaki}. Then $|\varphi|(a)= (u^*\varphi)(a)= \varphi(u^*a) = \varphi (a^\frac{1}{2}u^*a^\frac{1}{2}) = 
(a^\frac{1}{2} \varphi a^\frac{1}{2})(u^* ) \leq \|a^\frac{1}{2} \varphi a^\frac{1}{2}\|$. 
On the other hand, $\|a^\frac{1}{2}\varphi a^\frac{1}{2}\| = 
\|a^\frac{1}{2}u |\varphi| a^\frac{1}{2}\| = \|u a^\frac{1}{2}|\varphi| a^\frac{1}{2} \| \leq 
\|a^\frac{1}{2}|\varphi| a^\frac{1}{2}\| = |\varphi|(a).$ Hence $(ix)$ is satisfied.
 
\smallskip
$(ix) \Rightarrow (x)$. If $(ix)$ holds true, then we have
$|\varphi + \psi |(a) = \| a^\frac{1}{2} (\varphi + \psi) a^\frac{1}{2} \| \le 
\| a^\frac{1}{2} \varphi a^\frac{1}{2} \| + \| a^\frac{1}{2} \psi a^\frac{1}{2} \| =
| \varphi |(a) + |\psi |(a) .
$

\smallskip
Of course, $(x)$ implies $(vi)$. 
\end{proof}

\begin{remark}
From the preceding proof, it is seen that we can add the following condition to the list 
$(i)$ -- $(x)$ of equivalent conditions.
\begin{itemize}
 
\item [$(xi)$] 
{\it $|\varphi|(a)=\|a^\frac{1}{2}\varphi a^\frac{1}{2}\|$ for all $\varphi$ in $\mathcal{M}_*^h$}.  
\end{itemize}
  
\end{remark}

The following theorem is an analog of Gardner's characterization of traces by ``triangle inequality'' \cite{Gardner} (see, also, \cite{P&S}).

\begin{theorem}\label{gardner_crit} 
 A positive element $a$ of $\mathcal{M}$ belongs to the center of $\mathcal{M}$ if and only 
 if the inequality  $|\varphi(a)|\leq|\varphi|(a)$ holds for any $\varphi$ in $\mathcal{M}_*$.
\end{theorem}
\begin{proof}
Taking into account Corollary \ref{cor1}, for $a \in \mathcal{Z}$ and $\varphi \in \mathcal{M}_*$, we get 
$|\varphi(a)| = | (a^\frac{1}{2}\varphi a^\frac{1}{2})(\mathbf{1})| \leq \|a^\frac{1}{2}\varphi a^\frac{1}{2}\|=|\varphi|(a)$.

Let $|\varphi(a)|\leq|\varphi|(a)$ for any $\varphi$ in $\mathcal{M}_*$, then $|(u\psi)(a)|=|\psi(ua)|\leq \psi(a)$ for any unitary $u$ and any positive normal $\psi$.
Since the unit ball of $\mathcal{M}$ is the closed convex hull of the set of unitaries \cite{R&D} (see also \cite[1.1.12]{Pedersen}), it follows 
that $|\psi(xa)|\leq \|x\|\psi(a)$ for every $x$ in $\mathcal{M}$. But then the functional $\psi a$ attains
its norm at $\mathbf{1}$ and is therefore positive \cite[2.1.9]{Dixmier}. If $x \in \mathcal{M}^{sa}$
then $\psi(xa) = \overline{\psi(xa)} = \psi(ax)$. It follows that $xa=ax$ for any $x \in \mathcal{M}^{sa}$ 
and therefore $a$ belongs to the center.
\end{proof}
 
\begin{corollary} \label{cor2}
A positive element $a$ of a C*-algebra $\mathcal{A}$ belongs to the center of $\mathcal{A}$ if and only if the inequality  $|\varphi(a)|\leq|\varphi|(a)$ holds for any $\varphi$ in $\mathcal{A}^*$.
\end{corollary}

\begin{proof}
Let $(\pi , \mathfrak{H})$ be the universal representation of $\mathcal{A}$ and 
$\mathcal{M}(\pi )$ be the universal enveloping von Neumann algebra \cite[Section III.2]{Takesaki}. 
By construction of $\mathcal{M}(\pi )$, the spaces $\mathcal{A}^*$ and $\mathcal{M}(\pi )_*$ are isometrically isomorphic in a natural way. For $\varphi \in \mathcal{A}^*$, we will denote by $\widetilde \varphi$ the corresponding functional in $\mathcal{M}(\pi )_*$. Note that 
$\widetilde{ |\varphi |} = |\widetilde \varphi |$ by construction of absolute value \cite[12.2.7, 12.2.8]{Dixmier}.

Let $a$ be a positive element of the center of $\mathcal{A}$. It is easy to see that $\pi (a)$ belongs to the center of $\mathcal{M}(\pi )$. Then for any $\varphi \in \mathcal{A}^*$ we have 
$| \varphi (a)| = |\widetilde{\varphi} (\pi (a))| \le |\widetilde{\varphi}| (\pi (a)) =
\widetilde{ |\varphi |} (\pi (a)) = |\varphi | (a)$. 
 
On the other hand, let $a$ be a positive element of $\mathcal{A}$ and the inequality  
$|\varphi(a)| \leq |\varphi|(a)$ hold for any $\varphi$ in $\mathcal{A}^*$. Then 
$|\widetilde{\varphi}(\pi (a))| \leq |\widetilde{\varphi}|(\pi (a))$ for any 
$\widetilde{\varphi}$ in $\mathcal{M}(\pi )_*$, which implies that $\pi (a)$ lies in the center of  
$\mathcal{M}(\pi )$, hence $a$ lies in the center of $\mathcal{A}$.
\end{proof}

\begin{remark}
One can easily adapt the conditions $(iii)$ -- $(xi)$ to the case of C*-algebras and see that each of those modified conditions characterizes central elements.  
\end{remark}

\end{document}